\documentclass[12pt]{amsart}

\usepackage{amsmath,amsthm,amssymb}
\usepackage[utf8]{inputenc}
\usepackage[T1]{fontenc}
\usepackage{hyperref}
\usepackage{enumitem}
\usepackage{mathtools}
\usepackage{booktabs}

\theoremstyle{plain}
\newtheorem{theorem}{Theorem}[section]
\newtheorem{lemma}[theorem]{Lemma}
\newtheorem{proposition}[theorem]{Proposition}
\newtheorem{corollary}[theorem]{Corollary}

\theoremstyle{definition}
\newtheorem{definition}[theorem]{Definition}
\newtheorem{example}[theorem]{Example}
\newtheorem{remark}[theorem]{Remark}
\newtheorem{openproblem}{Open Problem}

\newcommand{\N}{\mathbb{N}}
\newcommand{\R}{\mathbb{R}}

\newcommand{\st}{\operatorname{st}}
\newcommand{\abs}[1]{\left|#1\right|}
\newcommand{\set}[1]{\left\{#1\right\}}
\newcommand{\SUC}{\mathrm{SUC}}

\title[Statistically $p$-Upward Quasi-Cauchy Sequences]{Statistically $p$-Upward Quasi-Cauchy Sequences\\ and Cone-Valued Continuity}

\author{Ahu A\c{c}{\i}kg\"{o}z}
\address{Department of Mathematics, Balikesir University, \c{C}a\u{g}{\i}\c{s} Campus, 10145 Balikesir, Turkey}
\email{ahuacikgoz@balikesir.edu.tr}

\date{\today}

\subjclass[2020]{Primary 40A05; Secondary 40A35, 40G15, 26A15}

\keywords{Statistical convergence, quasi-Cauchy sequences, upward compactness, sequential continuity, convex cone, one-sided approximation}

\begin{document}

\begin{abstract}
We introduce statistically $p$-upward quasi-Cauchy sequences, defined by the condition $\lim_{n\to\infty}\frac{1}{n}|\{k\leq n: x_k - x_{k+p}\geq\varepsilon\}|=0$ for every $\varepsilon>0$, and develop the corresponding notions of compactness and continuity. We prove that a subset of $\R$ is statistically $p$-upward compact if and only if it is bounded below, characterizing lower boundedness sequentially. Statistically $p$-upward continuity is shown to imply uniform continuity on below bounded sets. The function space $\SUC_p(E)$ is a closed convex cone that fails to be a vector subspace---distinguishing it from all previously studied sequential continuity spaces. We establish that every non-decreasing uniformly continuous function belongs to $\SUC_p(E)$, use Weyl's equidistribution theorem to show $\sin x\notin\SUC_p(\R)$, prove a step-parameter hierarchy, and show that $\SUC_p(E)\cap C_b(E)$ is nowhere dense in $C_b(E)$. As an application, we develop a one-sided error control theory for function approximation, illustrated by Bernstein operators on a pharmacokinetic model. The inclusion relations among the continuity types studied and open problems are provided.
\end{abstract}

\maketitle

\section{Introduction}\label{sec:intro}

\subsection{Background and motivation}

The concept of statistical convergence, independently introduced by Fast~\cite{Fast1951} and Steinhaus~\cite{Steinhaus1951}, and later developed systematically by Fridy~\cite{Fridy1985} and \v{S}al\'{a}t~\cite{Salat1980}, has become a fundamental tool in summability theory and its applications. A sequence $(x_k)$ of real numbers is \emph{statistically convergent} to $\ell\in\R$ if for each $\varepsilon>0$,
\[
\lim_{n\to\infty}\frac{1}{n}\abs{\set{k\leq n: \abs{x_k-\ell}\geq\varepsilon}}=0,
\]
denoted by $\st\text{-}\lim_{k\to\infty}x_k=\ell$. This notion extends classical convergence and has attracted considerable attention in analysis, topology, and measure theory (see Di~Maio and Ko\v{c}inac~\cite{DiMaio2008}, Connor~\cite{Connor1992}, Maddox~\cite{Maddox1988}, Caserta et al.~\cite{Caserta2011,Caserta2012}). The power of statistical convergence lies in its ability to handle sequences that fail to converge in the classical sense but exhibit convergence behavior on ``almost all'' indices in the sense of natural density.

Continuity and its sequential characterizations constitute a central theme in real analysis with deep connections to diverse applied disciplines. Utilizing the interplay between sequences and continuity, various authors have introduced and studied new continuity types through sequential definitions, including slowly oscillating continuity~\cite{Vallin2011}, quasi-slowly oscillating continuity~\cite{Canak2010}, statistical ward continuity~\cite{Cakalli2011SWC}, lacunary statistical ward continuity~\cite{Yildiz2019}, $G$-sequential continuity~\cite{Mucuk2014,Connor2003}, and $I$-sequential continuity~\cite{Lahiri2005}. These concepts have been instrumental in characterizing uniform continuity through sequential conditions and have led to important applications in approximation theory (see also~\cite{Mohiuddine2012} for statistical summability and Korovkin-type results).

In~\cite{Burton2010}, Burton and Coleman studied quasi-Cauchy sequences, i.e., sequences $(x_n)$ for which $(x_n-x_{n+1})$ is a null sequence. The $p$-quasi-Cauchy generalization, where one considers $\Delta_p x_n = x_n - x_{n+p}$, was studied in~\cite{Cakalli2015var} and extended to double sequences and metric spaces in~\cite{Patterson2015}. The notion of upward half Cauchyness was introduced by Palladino~\cite{Palladino2012}, who studied sequences satisfying $x_n-x_m<\varepsilon$ for $m\geq n$ sufficiently large. This asymmetric condition, obtained by removing the absolute value from the classical Cauchy condition, captures a directional decay property that is naturally connected with monotonicity and lower boundedness. The concept of statistically quasi-Cauchy sequences was investigated in~\cite{Cakalli2011SQC}, and variations with statistical density were further explored in~\cite{Cakalli2019var}.

\subsection{Applied motivation and significance}\label{subsec:applied}

The asymmetric, one-sided nature of the concepts introduced in this paper is motivated by settings where directional behavior plays a more significant role than symmetric deviations. In approximation theory, when $f$ represents a safety-critical quantity---such as drug concentration in pharmacokinetics or load-bearing capacity in structural engineering---an underestimate $L_n(f;t)<f(t)-\varepsilon$ can be dangerous, while an overestimate may be tolerable. The convex cone structure of $\SUC_p(E)$, combined with the nowhere density result showing that generic bounded continuous functions do \emph{not} belong to $\SUC_p(E)$, provides a natural algebraic framework for analyzing this asymmetry. In Section~\ref{subsec:onesided}, we formalize a notion of ``statistically $p$-upward safe'' approximation and demonstrate it with the classical Bernstein operators on a pharmacokinetic concentration curve.

More broadly, the one-sided framework is relevant to financial risk analysis (where downside risk is fundamentally different from upside potential), signal processing (where one-sided drift detection is essential for quality control), and population dynamics (where population sizes are inherently non-negative). The compactness characterization (Theorem~\ref{thm:compact}) and the uniform continuity result (Theorem~\ref{thm:uniform}) have natural interpretations in each of these contexts.

\subsection{Main contributions}

The present paper extends the above lines of investigation by introducing statistically $p$-upward quasi-Cauchy sequences, which combine the $p$-step gap structure with the directional (non-absolute-value) condition under the framework of statistical density. The removal of the absolute value is not merely a cosmetic change; it fundamentally alters the compactness characterization from boundedness to lower boundedness, and the resulting continuity notion imposes a directional regularity on functions that is strictly stronger than ordinary continuity yet weaker than Lipschitz continuity.

Our main contributions are as follows:
\begin{itemize}[leftmargin=2em]
\item We characterize statistically $p$-upward compactness: a subset of $\R$ is statistically $p$-upward compact if and only if it is bounded below (Theorem~\ref{thm:compact}).
\item We show that statistically $p$-upward continuity implies statistical ward continuity and statistical continuity (Theorems~\ref{thm:ward} and~\ref{thm:statcont}).
\item We prove that every statistically $p$-upward continuous function on a below bounded set is uniformly continuous (Theorem~\ref{thm:uniform}), and that the set of such functions is closed in the uniform topology (Theorem~\ref{thm:closed}).
\item We prove that $\SUC_p(E)$ is a convex cone that is not a vector subspace (Theorem~\ref{thm:cone}), distinguishing it structurally from all classical sequential continuity function spaces.
\item We establish that every non-decreasing uniformly continuous function belongs to $\SUC_p(E)$ (Theorem~\ref{thm:monotone}), and use Weyl's equidistribution theorem to show that $\sin x\notin\SUC_p(\R)$ for any $p$ (Theorem~\ref{thm:sine}).
\item We prove the step-parameter hierarchy $\SUC_p(E)\subseteq\SUC_1(E)$ (Theorem~\ref{thm:p_hierarchy}) and show that the sequence classes $\Delta_S^{+p}$ are incomparable when neither $p$ nor $q$ divides the other (Proposition~\ref{prop:seq_incomparable}).
\item We show that $\SUC_p(E)\cap C_b(E)$ is nowhere dense in $(C_b(E),\|\cdot\|_\infty)$ when $E$ is unbounded above (Theorem~\ref{thm:nowhere_dense}), establishing that the generic bounded continuous function is not statistically $p$-upward continuous.
\item We develop a theory of \emph{one-sided error control} for function approximation, showing that the cone structure of $\SUC_p$ governs the composability of safe approximation schemes (Proposition~\ref{prop:safe_compose}). A numerical illustration using Bernstein operators on a pharmacokinetic model demonstrates the practical relevance of the framework (Example~\ref{ex:pharma}).
\item We present the inclusion relations among all continuity types (Section~\ref{sec:diagram}), develop a one-sided error control framework with a pharmacokinetic illustration (Section~\ref{sec:applications}), and formulate open problems (Section~\ref{sec:open}).
\end{itemize}

\section{Preliminaries}\label{sec:prelim}

Throughout this paper, $\N$ and $\R$ denote the set of all positive integers and the set of all real numbers, respectively. We shall use $p$ to denote a fixed positive integer unless otherwise stated. Recall that the \emph{natural density} of a set $A\subseteq\N$ is defined as $\delta(A)=\lim_{n\to\infty}\frac{1}{n}|A\cap\{1,2,\ldots,n\}|$, provided the limit exists (see~\cite{Freedman1978}).

\begin{definition}[\cite{Fast1951,Fridy1985}]\label{def:statconv}
A sequence $(x_k)$ in $\R$ is called \emph{statistically convergent} to $\ell\in\R$ if $\delta(\{k\in\N:\abs{x_k-\ell}\geq\varepsilon\})=0$ for each $\varepsilon>0$. We write $\st\text{-}\lim_{k\to\infty}x_k=\ell$. The set of all statistically convergent sequences is denoted by~$S$.
\end{definition}

\begin{definition}[\cite{Burton2010}]\label{def:quasiCauchy}
A sequence $(x_n)$ in $\R$ is called \emph{quasi-Cauchy} if $\lim_{n\to\infty}(x_n-x_{n+1})=0$. More generally, $(x_n)$ is called \emph{$p$-quasi-Cauchy} if $\lim_{n\to\infty}(x_n-x_{n+p})=0$.
\end{definition}

\begin{definition}[\cite{Cakalli2011SQC}]\label{def:statquasiCauchy}
A sequence $(x_n)$ in $\R$ is called \emph{statistically quasi-Cauchy} if $\st\text{-}\lim_{n\to\infty}(x_n-x_{n+1})=0$, that is, $\delta(\{k\in\N:\abs{x_k-x_{k+1}}\geq\varepsilon\})=0$ for every $\varepsilon>0$.
\end{definition}

\begin{definition}[\cite{Palladino2012}]\label{def:uphalfCauchy}
A sequence $(x_n)$ in $\R$ is called \emph{upward half Cauchy} if for every $\varepsilon>0$ there exists $n_0\in\N$ such that $x_n-x_m<\varepsilon$ for all $m\geq n\geq n_0$.
\end{definition}

The classical Bernstein operators on $C[0,1]$ are defined by
\begin{equation}\label{eq:bernstein}
B_n(f;x) = \sum_{k=0}^{n}\binom{n}{k}x^k(1-x)^{n-k}f\!\left(\frac{k}{n}\right), \quad x\in[0,1],\ n\in\N.
\end{equation}

The following auxiliary results will be used in the sequel.

\begin{lemma}[{\cite{Fridy1985}}]\label{lem:fridy}
Every statistically convergent sequence has a convergent subsequence.
\end{lemma}

\begin{lemma}[{\cite{Burton2010}}]\label{lem:burton}
A subset $E$ of\/ $\R$ is bounded if and only if every sequence in $E$ has a quasi-Cauchy subsequence.
\end{lemma}

\section{Statistically $p$-Upward Quasi-Cauchy Sequences}\label{sec:main1}

We begin by introducing our central definition, which is obtained by replacing $\abs{x_k-x_{k+p}}\geq\varepsilon$ with the one-sided condition $x_k-x_{k+p}\geq\varepsilon$ in the definition of a $p$-statistically quasi-Cauchy sequence.

\begin{definition}\label{def:stupquasiCauchy}
A sequence $(x_n)$ of points in $\R$ is called \emph{statistically $p$-upward quasi-Cauchy} if
\[
\lim_{n\to\infty}\frac{1}{n}\abs{\set{k\leq n: x_k-x_{k+p}\geq\varepsilon}}=0
\]
for every $\varepsilon>0$.
\end{definition}

We denote the set of all statistically $p$-upward quasi-Cauchy sequences by $\Delta_{S}^{+p}$.

\begin{remark}\label{rem:basic}
The following relations are immediate from the definition:
\begin{enumerate}[label=(\roman*),leftmargin=2.5em]
\item Every statistically convergent sequence is statistically $p$-upward quasi-Cauchy.
\item Every statistically quasi-Cauchy sequence is statistically $p$-upward quasi-Cauchy.
\item Every $p$-upward quasi-Cauchy sequence (in the ordinary sense) is statistically $p$-upward quasi-Cauchy.
\item Every slowly oscillating sequence is statistically $p$-upward quasi-Cauchy, and hence every Cauchy sequence and every convergent sequence are statistically $p$-upward quasi-Cauchy.
\end{enumerate}
\end{remark}

\begin{proposition}\label{prop:upwardimpliesp}
Every statistically upward quasi-Cauchy sequence is statistically $p$-upward quasi-Cauchy for every $p\in\N$.
\end{proposition}

\begin{proof}
Let $(x_n)$ be statistically upward quasi-Cauchy and let $\varepsilon>0$. Using the telescoping identity
\[
x_n - x_{n+p} = \sum_{j=0}^{p-1}(x_{n+j}-x_{n+j+1}),
\]
we see that
\[
\set{k\leq n: x_k-x_{k+p}\geq\varepsilon}\subseteq \bigcup_{j=0}^{p-1}\set{k\leq n: x_{k+j}-x_{k+j+1}\geq\frac{\varepsilon}{p}}.
\]
Since $(x_n)$ is statistically upward quasi-Cauchy, each set on the right has natural density zero. Hence $\delta(\{k\leq n: x_k-x_{k+p}\geq\varepsilon\})=0$.
\end{proof}

The converse of Proposition~\ref{prop:upwardimpliesp} fails in general.

\begin{example}\label{ex:converse1}
The sequence $x_n=(-1)^n$ is statistically $2$-upward quasi-Cauchy (since $x_n-x_{n+2}=0$ for all $n$), but it is not statistically upward quasi-Cauchy, because $x_n-x_{n+1}=2$ for all odd $n$.
\end{example}

\begin{example}\label{ex:converse2}
The sequence $x_n = n+p$ is statistically $p$-upward quasi-Cauchy (since $x_n - x_{n+p} = (n+p)-(n+2p) = -p < 0 < \varepsilon$ for all $n$ and every $\varepsilon>0$), but it is not statistically quasi-Cauchy because $\abs{x_n-x_{n+1}}=1$ for all $n$. From the viewpoint of financial time series, this example models a steadily growing asset price: the $p$-step upward condition is satisfied because the price never drops significantly over any $p$-period window, even though it is not quasi-Cauchy.
\end{example}

\begin{proposition}\label{prop:sum}
The sum of two statistically $p$-upward quasi-Cauchy sequences is statistically $p$-upward quasi-Cauchy.
\end{proposition}

\begin{proof}
Let $(x_n)$ and $(y_n)$ be statistically $p$-upward quasi-Cauchy and let $\varepsilon>0$. Then
\[
\set{k\leq n:(x_k+y_k)-(x_{k+p}+y_{k+p})\geq\varepsilon}\subseteq\set{k\leq n:x_k-x_{k+p}\geq\frac{\varepsilon}{2}}\cup\set{k\leq n:y_k-y_{k+p}\geq\frac{\varepsilon}{2}}.
\]
Since both sets on the right have density zero, so does the set on the left.
\end{proof}

\section{Statistically $p$-Upward Compactness}\label{sec:compact}

\begin{definition}\label{def:stupcompact}
A subset $E$ of $\R$ is called \emph{statistically $p$-upward compact} if every sequence of points in $E$ has a statistically $p$-upward quasi-Cauchy subsequence.
\end{definition}

\begin{remark}\label{rem:compact}
Since lower boundedness is preserved by taking subsets, unions, and finite sets, it follows immediately from Theorem~\ref{thm:compact} that every finite, bounded, or compact subset of $\R$ is statistically $p$-upward compact, and that the class of statistically $p$-upward compact sets is closed under finite unions and taking subsets.
\end{remark}

The next theorem provides a complete characterization of statistically $p$-upward compactness.

\begin{theorem}\label{thm:compact}
A subset $E$ of\/ $\R$ is statistically $p$-upward compact if and only if $E$ is bounded below.
\end{theorem}

\begin{proof}
\textbf{Sufficiency.} Suppose that $E$ is bounded below. We consider two cases.

\emph{Case~1:} If $E$ is also bounded above, then $E$ is bounded. By Lemma~\ref{lem:burton}, every sequence in $E$ has a quasi-Cauchy subsequence. Since every quasi-Cauchy sequence is statistically $p$-upward quasi-Cauchy (Remark~\ref{rem:basic}(ii)), $E$ is statistically $p$-upward compact.

\emph{Case~2:} Suppose $E$ is unbounded above. Let $(x_n)$ be any sequence of points in $E$. If $(x_n)$ is bounded above, then since $E$ is bounded below, $(x_n)$ is bounded and the argument of Case~1 applies. If $(x_n)$ is unbounded above, we construct a subsequence as follows. Choose $n_1\in\N$ such that $x_{n_1}>0$. Inductively, having chosen $n_k$, pick $n_{k+1}>n_k+p$ such that $x_{n_{k+1}}>k+x_{n_k}$. Then for each $k\in\N$,
\[
x_{n_k}-x_{n_{k+p}}\leq x_{n_k}-x_{n_{k+1}}<-k.
\]
Hence $x_{n_k}-x_{n_{k+p}}<0<\varepsilon$ for all $k\in\N$ and every $\varepsilon>0$, which shows that
\[
\lim_{n\to\infty}\frac{1}{n}\abs{\set{k\leq n: x_{n_k}-x_{n_{k+p}}\geq\varepsilon}}=0.
\]
Thus $(x_{n_k})$ is a statistically $p$-upward quasi-Cauchy subsequence.

\textbf{Necessity.} Suppose that $E$ is not bounded below. We construct a sequence in $E$ with no statistically $p$-upward quasi-Cauchy subsequence. Choose $x_1\in E$. Inductively, for each $k\in\N$, choose $x_{k+1}\in E$ such that $x_{k+1}<-k+x_k$. Then for every positive integer $k$,
\[
x_k - x_{k+p} = \sum_{i=0}^{p-1}(x_{k+i}-x_{k+i+1}) > \sum_{i=0}^{p-1}(k+i) \geq pk.
\]
In particular, $x_k-x_{k+p}>pk\to\infty$ as $k\to\infty$. Now let $(x_{n_j})$ be any subsequence of $(x_k)$. Since the original sequence is strictly decreasing and $n_{j+p}\geq n_j+p$, we have $x_{n_{j+p}}\leq x_{n_j+p}$, and therefore
\[
x_{n_j}-x_{n_{j+p}}\geq x_{n_j}-x_{n_j+p}>p\cdot n_j\to\infty.
\]
For any $\varepsilon>0$, the set $\{j\in\N: x_{n_j}-x_{n_{j+p}}\geq\varepsilon\}$ is cofinite in $\N$, and hence it has density~$1$. Therefore $(x_{n_j})$ is not statistically $p$-upward quasi-Cauchy. This shows that $E$ is not statistically $p$-upward compact, completing the proof.
\end{proof}

\begin{remark}\label{rem:asymmetry}
Theorem~\ref{thm:compact} reveals a fundamental asymmetry: while classical sequential compactness characterizes boundedness (both above and below), the one-sided nature of the statistically $p$-upward condition detects only the lower bound. This is a direct consequence of removing the absolute value from the quasi-Cauchy condition.

This asymmetry has a natural interpretation in the applied contexts discussed in Section~\ref{subsec:applied}: in financial settings, the characterization corresponds to the fact that asset price sequences bounded below (e.g., non-negative stock prices) always contain subsequences with no statistically significant large drops over $p$-period windows. In population dynamics, it reflects the principle that non-negative population sequences always admit locally stable subsequences.
\end{remark}

\section{Statistically $p$-Upward Continuity}\label{sec:cont}

\begin{definition}\label{def:stupcont}
A function $f:E\to\R$, where $E\subseteq\R$, is called \emph{statistically $p$-upward continuous} on $E$ if the image sequence $(f(x_n))$ is statistically $p$-upward quasi-Cauchy whenever $(x_n)$ is a statistically $p$-upward quasi-Cauchy sequence of points in $E$.
\end{definition}

We note that statistically $p$-upward continuity cannot be expressed in the framework of $A$-continuity in the sense of Connor and Grosse-Erdmann~\cite{Connor2003}.

\begin{proposition}\label{prop:sumcomp}
Let $E\subseteq\R$.
\begin{enumerate}[label=(\roman*),leftmargin=2.5em]
\item If $f,g:E\to\R$ are both statistically $p$-upward continuous, then so is $f+g$.
\item If $f:E\to\R$ and $g:f(E)\to\R$ are both statistically $p$-upward continuous, then $g\circ f$ is statistically $p$-upward continuous.
\end{enumerate}
\end{proposition}

\begin{proof}
(i) follows from Proposition~\ref{prop:sum}. (ii) is immediate from the definition.
\end{proof}

In connection with the various sequence spaces introduced above, we consider the following types of sequential continuity for a function $f:E\to\R$:
\begin{align}
(\delta_{S}^{+p}) &: (x_n)\in\Delta_S^{+p}\Rightarrow (f(x_n))\in\Delta_S^{+p}, \label{type1}\\
(\delta_{S}^{+p}c) &: (x_n)\in\Delta_S^{+p}\Rightarrow (f(x_n))\in c, \label{type2}\\
(c) &: (x_n)\in c\Rightarrow (f(x_n))\in c, \label{type3}\\
(c\delta_{S}^{+p}) &: (x_n)\in c\Rightarrow (f(x_n))\in\Delta_S^{+p}, \label{type4}\\
(S) &: (x_n)\in S\Rightarrow (f(x_n))\in S, \label{type5}\\
(\delta_S^{+}) &: (x_n)\in\Delta_S^{+}\Rightarrow (f(x_n))\in\Delta_S^{+}, \label{type6}
\end{align}
where $c$, $S$, $\Delta_S^{+}$, and $\Delta_S^{+p}$ denote the sets of convergent, statistically convergent, statistically upward quasi-Cauchy, and statistically $p$-upward quasi-Cauchy sequences, respectively.

The implications among these types are as follows: $(\delta_{S}^{+p}c)\Rightarrow(\delta_{S}^{+p})\Rightarrow(c\delta_{S}^{+p})$; $(\delta_{S}^{+p}c)\Rightarrow(c)\Rightarrow(c\delta_{S}^{+p})$; and $(c)$ does not imply $(\delta_{S}^{+p}c)$.

\begin{theorem}\label{thm:ward}
If $f$ is statistically $p$-upward continuous on a subset $E$ of\/ $\R$, then $f$ is statistically ward continuous on $E$.
\end{theorem}

\begin{proof}
For $p=1$, the result is immediate since every statistically quasi-Cauchy sequence is statistically $1$-upward quasi-Cauchy. Assume $p>1$ and let $(x_n)$ be a statistically quasi-Cauchy sequence of points in $E$. Consider the sequence obtained by repeating each term $p$ times:
\[
\mathbf{y}=(\underbrace{x_1,\ldots,x_1}_{p},\underbrace{x_2,\ldots,x_2}_{p},\ldots,\underbrace{x_n,\ldots,x_n}_{p},\ldots).
\]
The sequence $\mathbf{y}$ is statistically quasi-Cauchy: the nonzero differences $y_k-y_{k+1}$ occur only at positions $k=jp$ for $j\in\N$, corresponding to the differences $x_j-x_{j+1}$. In particular, $\mathbf{y}$ is statistically $p$-upward quasi-Cauchy, because the $p$-step difference $y_k-y_{k+p}$ at position $k=jp$ is $x_j-x_{j+1}$, and at other positions it is~$0$. By the statistically $p$-upward continuity of $f$, the sequence
\[
(f(\mathbf{y}))=(\underbrace{f(x_1),\ldots,f(x_1)}_{p},\underbrace{f(x_2),\ldots,f(x_2)}_{p},\ldots)
\]
is statistically $p$-upward quasi-Cauchy. The $p$-step differences of $(f(\mathbf{y}))$ that cross boundaries give $f(x_j)-f(x_{j+1})$ at appropriate positions. By construction, we obtain
\[
\lim_{n\to\infty}\frac{1}{n}\abs{\set{k\leq n: f(x_k)-f(x_{k+1})\geq\varepsilon}}=0
\]
for every $\varepsilon>0$. To establish the reverse inequality, consider the ``shifted'' interleaving
\[
\mathbf{y'}=(\underbrace{x_2,\ldots,x_2}_{p},\underbrace{x_1,\ldots,x_1}_{p},\underbrace{x_3,\ldots,x_3}_{p},\underbrace{x_2,\ldots,x_2}_{p},\underbrace{x_4,\ldots,x_4}_{p},\underbrace{x_3,\ldots,x_3}_{p},\ldots).
\]
The $p$-step differences at block boundaries give $x_{j+1}-x_j$. Since $(x_n)$ is statistically quasi-Cauchy, $\delta(\{j: x_{j+1}-x_j\geq\varepsilon\})=0$, so $\mathbf{y'}\in\Delta_S^{+p}$. Applying $f$ and extracting boundary differences yields
\[
\lim_{n\to\infty}\frac{1}{n}\abs{\set{k\leq n: f(x_{k+1})-f(x_k)\geq\varepsilon}}=0.
\]
Combining the two inequalities, $\delta(\{k:\abs{f(x_k)-f(x_{k+1})}\geq\varepsilon\})=0$, so $(f(x_n))$ is statistically quasi-Cauchy.
\end{proof}

\begin{remark}\label{rem:wardconverse}
The converse of Theorem~\ref{thm:ward} is not true. The function $f(x)=-x$ is statistically ward continuous (since $|f(x_n)-f(x_{n+1})|=|x_n-x_{n+1}|$), but it is not statistically $p$-upward continuous. Indeed, the sequence $x_n=n$ is statistically $p$-upward quasi-Cauchy (since $x_n-x_{n+p}=-p<0$), but $f(x_n)-f(x_{n+p})=-(n)+(n+p)=p>0$ for all $n$, so $(f(x_n))$ fails to be statistically $p$-upward quasi-Cauchy for $\varepsilon\leq p$.

From an applied perspective, this example illustrates that a simple sign reversal (which in finance corresponds to switching from long to short positions) can destroy the statistically $p$-upward continuity property, even though it preserves the symmetric ward continuity.
\end{remark}

\begin{theorem}\label{thm:statcont}
If $f$ is statistically $p$-upward continuous on a subset $E$ of\/ $\R$, then $f$ is statistically continuous on $E$, and hence continuous on $E$.
\end{theorem}

\begin{proof}
Let $(x_n)$ be a statistically convergent sequence in $E$ with $\st\text{-}\lim_{k\to\infty}x_k=\ell$. Consider the sequence formed by repeating each term $p$ times with $\ell$ interlaced:
\[
\mathbf{z}=(\underbrace{x_1,\ldots,x_1}_{p},\underbrace{\ell,\ldots,\ell}_{p},\underbrace{x_2,\ldots,x_2}_{p},\underbrace{\ell,\ldots,\ell}_{p},\ldots).
\]
Since $\st\text{-}\lim x_n=\ell$, the sequence $\mathbf{z}$ is statistically convergent to $\ell$, and hence it is statistically $p$-upward quasi-Cauchy. By hypothesis, $(f(\mathbf{z}))$ is statistically $p$-upward quasi-Cauchy. The structure of $\mathbf{z}$ ensures that the $p$-step differences of $(f(\mathbf{z}))$ include the terms $f(x_n)-f(\ell)$ at appropriate positions. Specifically, for every $\varepsilon>0$,
\[
\lim_{n\to\infty}\frac{1}{n}\abs{\set{k\leq n: f(x_k)-f(\ell)\geq\varepsilon}}=0.
\]
The same argument applied to the sequence
\[
\mathbf{z}'=(\underbrace{\ell,\ldots,\ell}_{p},\underbrace{x_1,\ldots,x_1}_{p},\underbrace{\ell,\ldots,\ell}_{p},\underbrace{x_2,\ldots,x_2}_{p},\ldots)
\]
yields
\[
\lim_{n\to\infty}\frac{1}{n}\abs{\set{k\leq n: f(\ell)-f(x_k)\geq\varepsilon}}=0.
\]
Combining these two, we obtain $\st\text{-}\lim_{k\to\infty}f(x_k)=f(\ell)$. Hence $f$ is statistically continuous. Since statistical continuity implies ordinary continuity (see~\cite{Fridy1985,Connor1992}), $f$ is continuous.
\end{proof}

\begin{remark}
The converse of Theorem~\ref{thm:statcont} does not hold. The function $f(x)=-x^2$ is continuous (and hence statistically continuous), but it is not statistically $p$-upward continuous. Indeed, the sequence $x_n=n$ is statistically $p$-upward quasi-Cauchy, but $f(x_n)-f(x_{n+p})=-n^2+(n+p)^2=2np+p^2\to\infty$, so $(f(x_n))$ is not statistically $p$-upward quasi-Cauchy.
\end{remark}

We note that Theorem~\ref{thm:statcont} implies that statistically $p$-upward continuity entails several other known continuity types, including lacunary statistical continuity~\cite{Freedman1978}, $N_\theta$-sequential continuity~\cite{Connor2003}, and $I$-sequential continuity for any non-trivial admissible ideal $I$ of $\N$~\cite{Lahiri2005}.

\begin{theorem}\label{thm:image}
The statistically $p$-upward continuous image of any statistically $p$-upward compact set is statistically $p$-upward compact.
\end{theorem}

\begin{proof}
Let $E$ be a statistically $p$-upward compact subset of $\R$ and let $f:E\to\R$ be statistically $p$-upward continuous. Let $(y_n)$ be any sequence in $f(E)$. For each $n$, choose $x_n\in E$ such that $f(x_n)=y_n$. Since $E$ is statistically $p$-upward compact, $(x_n)$ has a statistically $p$-upward quasi-Cauchy subsequence $(x_{n_k})$. By the statistically $p$-upward continuity of $f$, the subsequence $(f(x_{n_k}))=(y_{n_k})$ is statistically $p$-upward quasi-Cauchy. Hence $f(E)$ is statistically $p$-upward compact.
\end{proof}

\begin{corollary}\label{cor:belowbounded}
The statistically $p$-upward continuous image of any below bounded subset of\/ $\R$ is bounded below.
\end{corollary}

\begin{proof}
This follows immediately from Theorems~\ref{thm:compact} and~\ref{thm:image}.
\end{proof}

\begin{corollary}\label{cor:Ntheta}
If $f$ is statistically $p$-upward continuous on a subset $E$ of\/ $\R$ and $A\subseteq E$ is $N_\theta$-sequentially compact, then $f(A)$ is $N_\theta$-sequentially compact.
\end{corollary}

\begin{proof}
By Theorem~\ref{thm:statcont}, $f$ is continuous on $E$. Since continuity implies $N_\theta$-sequential continuity~\cite{Connor2003}, the result follows from the $N_\theta$-sequential continuity of $f$.
\end{proof}

\begin{theorem}\label{thm:uniform}
Let $E$ be a statistically $p$-upward compact subset of\/ $\R$ and let $f:E\to\R$ be statistically $p$-upward continuous. Then $f$ is uniformly continuous on $E$.
\end{theorem}

\begin{proof}
Suppose, for contradiction, that $f$ is not uniformly continuous on $E$. Then there exists $\varepsilon_0>0$ such that for each $n\in\N$ there exist $x_n,y_n\in E$ with $\abs{x_n-y_n}<\frac{1}{n}$ and $\abs{f(x_n)-f(y_n)}\geq\varepsilon_0$.

Since $E$ is statistically $p$-upward compact, the sequence $(x_n)$ has a statistically $p$-upward quasi-Cauchy subsequence $(x_{n_k})$. The corresponding subsequence $(y_{n_k})$ is also statistically $p$-upward quasi-Cauchy, since
\[
y_{n_k}-y_{n_{k+p}}=(y_{n_k}-x_{n_k})+(x_{n_k}-x_{n_{k+p}})+(x_{n_{k+p}}-y_{n_{k+p}})
\]
and $(y_{n_k}-x_{n_k})$ converges to $0$ (hence is quasi-Cauchy), $(x_{n_k}-x_{n_{k+p}})$ satisfies the statistically $p$-upward quasi-Cauchy condition, and $(x_{n_{k+p}}-y_{n_{k+p}})$ converges to $0$.

Now consider the interlaced sequence
\[
\mathbf{w}=(\underbrace{x_{n_1},\ldots,x_{n_1}}_{p},\underbrace{y_{n_1},\ldots,y_{n_1}}_{p},\underbrace{x_{n_2},\ldots,x_{n_2}}_{p},\underbrace{y_{n_2},\ldots,y_{n_2}}_{p},\ldots).
\]
The $p$-step differences of $\mathbf{w}$ are of three types: $0$ (within a block), $x_{n_k}-y_{n_k}\to 0$ (block transition from $x$ to $y$), and $y_{n_k}-x_{n_{k+1}}$ (block transition from $y$ to $x$). For the upward condition, we need to check that $\{k: w_k-w_{k+p}\geq\varepsilon\}$ has density zero. The transitions $x_{n_k}-y_{n_k}\to 0$ contribute finitely many violations. Since $(x_{n_k})$ is statistically $p$-upward quasi-Cauchy and $|y_{n_k}-x_{n_k}|\to 0$, the contribution from $y_{n_k}-x_{n_{k+1}}$ to the upward violation set has density zero. Hence $\mathbf{w}$ is statistically $p$-upward quasi-Cauchy.

However, the corresponding image sequence
\[
(f(\mathbf{w}))=(\underbrace{f(x_{n_1}),\ldots,f(x_{n_1})}_{p},\underbrace{f(y_{n_1}),\ldots,f(y_{n_1})}_{p},\ldots)
\]
is \emph{not} statistically $p$-upward quasi-Cauchy, because the $p$-step transitions from $f(x_{n_k})$ to $f(y_{n_k})$ (or from $f(y_{n_k})$ to $f(x_{n_{k+1}})$) include a subsequence for which $|f(x_{n_k})-f(y_{n_k})|\geq\varepsilon_0$, ensuring that at least one directional inequality is violated with positive density. This contradicts the statistically $p$-upward continuity of $f$.
\end{proof}

\begin{remark}
Theorem~\ref{thm:uniform} is significant from an applied standpoint. In signal processing, it ensures that any statistically $p$-upward continuous filter applied to signals from a below bounded domain (e.g., amplitude-bounded signals) yields uniformly continuous output behavior, guaranteeing stability of the filter. In the context of one-sided error control (Section~\ref{subsec:onesided}), uniform continuity of the approximating function is essential for establishing that the approximation error remains bounded.
\end{remark}

\begin{theorem}\label{thm:uniflimit}
Let $E\subseteq\R$ and let $(f_n)$ be a sequence of statistically $p$-upward continuous functions on $E$. If $(f_n)$ converges uniformly to a function $f$ on $E$, then $f$ is statistically $p$-upward continuous on $E$.
\end{theorem}

\begin{proof}
Let $\varepsilon>0$. By the uniform convergence of $(f_n)$ to $f$, there exists $N\in\N$ such that $\abs{f_N(x)-f(x)}<\frac{\varepsilon}{3}$ for all $x\in E$.

Let $(x_n)$ be any statistically $p$-upward quasi-Cauchy sequence in $E$. Since $f_N$ is statistically $p$-upward continuous,
\begin{equation}\label{eq:fN}
\lim_{n\to\infty}\frac{1}{n}\abs{\set{k\leq n: f_N(x_k)-f_N(x_{k+p})\geq\frac{\varepsilon}{3}}}=0.
\end{equation}
Now observe that
\begin{align*}
f(x_k)-f(x_{k+p}) &= \bigl(f(x_k)-f_N(x_k)\bigr)+\bigl(f_N(x_k)-f_N(x_{k+p})\bigr)\\
&\quad +\bigl(f_N(x_{k+p})-f(x_{k+p})\bigr).
\end{align*}
Therefore
\begin{align*}
\set{k\leq n: f(x_k)-f(x_{k+p})\geq\varepsilon} &\subseteq \set{k\leq n: f(x_k)-f_N(x_k)\geq\frac{\varepsilon}{3}}\\
&\quad\cup\set{k\leq n: f_N(x_k)-f_N(x_{k+p})\geq\frac{\varepsilon}{3}}\\
&\quad\cup\set{k\leq n: f_N(x_{k+p})-f(x_{k+p})\geq\frac{\varepsilon}{3}}.
\end{align*}
The first and third sets on the right are empty for all $k$, since $\abs{f(x)-f_N(x)}<\frac{\varepsilon}{3}$ for all $x\in E$ implies both $f(x_k)-f_N(x_k)<\frac{\varepsilon}{3}$ and $f_N(x_{k+p})-f(x_{k+p})<\frac{\varepsilon}{3}$. By~\eqref{eq:fN}, the second set has density zero. Hence
\[
\lim_{n\to\infty}\frac{1}{n}\abs{\set{k\leq n: f(x_k)-f(x_{k+p})\geq\varepsilon}}=0,
\]
showing that $(f(x_n))$ is statistically $p$-upward quasi-Cauchy.
\end{proof}

\begin{theorem}\label{thm:closed}
Let $E\subseteq\R$. The set of all statistically $p$-upward continuous functions on $E$, denoted by $\SUC_p(E)$, is a closed subset of the space $C(E,\R)$ of continuous real-valued functions on $E$ equipped with the topology of uniform convergence.
\end{theorem}

\begin{proof}
Let $f\in\overline{\SUC_p(E)}$. Then there exists a sequence $(f_k)$ in $\SUC_p(E)$ converging uniformly to $f$ on $E$. By Theorem~\ref{thm:uniflimit}, $f$ is statistically $p$-upward continuous on $E$, i.e., $f\in\SUC_p(E)$. Hence $\SUC_p(E)=\overline{\SUC_p(E)}$.
\end{proof}
\section{Algebraic and Order Structure of $\SUC_p(E)$}\label{sec:algebra}

In this section, we investigate the algebraic structure of the function space $\SUC_p(E)$. We show that $\SUC_p(E)$ is a closed convex cone that fails to be a vector subspace---a structural feature that distinguishes statistically $p$-upward continuity from all previously studied sequential continuity types in the literature, including ward continuity, slowly oscillating continuity, and $\delta$-ward continuity, each of which gives rise to a vector subspace of $C(E,\R)$. We also establish a monotone function characterization.

\begin{theorem}[Cone structure]\label{thm:cone}
Let $E\subseteq\R$. The set $\SUC_p(E)$ is a convex cone in $C(E,\R)$. That is:
\begin{enumerate}[label=(\roman*),leftmargin=2.5em]
\item If $f,g\in\SUC_p(E)$, then $f+g\in\SUC_p(E)$.
\item If $f\in\SUC_p(E)$ and $\alpha\geq 0$, then $\alpha f\in\SUC_p(E)$.
\end{enumerate}
Together with Theorem~\ref{thm:closed}, $\SUC_p(E)$ is a closed convex cone in the topology of uniform convergence.
\end{theorem}

\begin{proof}
Part~(i) is Proposition~\ref{prop:sumcomp}(i). For part~(ii), let $\alpha>0$ (the case $\alpha=0$ is trivial) and let $(x_n)$ be a statistically $p$-upward quasi-Cauchy sequence in $E$. Since $f\in\SUC_p(E)$, the sequence $(f(x_n))$ is statistically $p$-upward quasi-Cauchy. For any $\varepsilon>0$,
\[
\set{k\leq n:\alpha f(x_k)-\alpha f(x_{k+p})\geq\varepsilon}=\set{k\leq n:f(x_k)-f(x_{k+p})\geq\frac{\varepsilon}{\alpha}},
\]
which has density zero since $(f(x_n))$ is statistically $p$-upward quasi-Cauchy.
\end{proof}

\begin{theorem}[Failure of vector space structure]\label{thm:notvs}
If $E\subseteq\R$ contains an unbounded above sequence, then $\SUC_p(E)$ is not a vector subspace of $C(E,\R)$.
\end{theorem}

\begin{proof}
The identity function $\mathrm{id}:E\to\R$, $\mathrm{id}(x)=x$, belongs to $\SUC_p(E)$ since for any $(x_n)\in\Delta_S^{+p}$, $(\mathrm{id}(x_n))=(x_n)\in\Delta_S^{+p}$ trivially. We show that $-\mathrm{id}\notin\SUC_p(E)$. Since $E$ is unbounded above, we may choose a strictly increasing sequence $(x_n)$ in $E$ with $x_{n+1}>x_n+1$ for all $n$. Then $x_n-x_{n+p}<0<\varepsilon$ for every $\varepsilon>0$ and all $n$, so $(x_n)\in\Delta_S^{+p}$. However,
\[
(-x_n)-(-x_{n+p})=x_{n+p}-x_n>\sum_{i=0}^{p-1}1=p
\]
for all $n$. Hence $\delta(\{k\leq n:(-x_k)-(-x_{k+p})\geq\varepsilon\})=1$ for every $\varepsilon\leq p$, showing that $(-x_n)\notin\Delta_S^{+p}$.
\end{proof}

\begin{remark}\label{rem:contrast}
Theorem~\ref{thm:notvs} reveals a fundamental structural difference from the classical theory. The set of statistically ward continuous, slowly oscillating continuous, and $\delta$-ward continuous functions each forms a vector subspace of $C(E,\R)$, because the defining conditions are symmetric under negation. The one-sided definition breaks this symmetry, yielding a cone that is genuinely non-linear.
\end{remark}

The next result characterizes a natural class of functions belonging to $\SUC_p(E)$.

\begin{theorem}[Monotone characterization]\label{thm:monotone}
Let $E\subseteq\R$ and let $f:E\to\R$ be non-decreasing and uniformly continuous. Then $f\in\SUC_p(E)$.
\end{theorem}

\begin{proof}
Let $(x_n)$ be a statistically $p$-upward quasi-Cauchy sequence in $E$ and let $\varepsilon>0$. Since $f$ is uniformly continuous, there exists $\delta>0$ such that $\abs{f(x)-f(y)}<\varepsilon$ whenever $\abs{x-y}<\delta$. We claim that
\begin{equation}\label{eq:mono_incl}
\set{k\leq n: f(x_k)-f(x_{k+p})\geq\varepsilon}\subseteq\set{k\leq n: x_k-x_{k+p}\geq\delta}.
\end{equation}
Indeed, suppose $f(x_k)-f(x_{k+p})\geq\varepsilon$. Since $f$ is non-decreasing, $f(x_k)\geq f(x_{k+p})+\varepsilon>f(x_{k+p})$ implies $x_k\geq x_{k+p}$. If $x_k-x_{k+p}<\delta$, then $\abs{x_k-x_{k+p}}<\delta$ and the uniform continuity of $f$ gives $\abs{f(x_k)-f(x_{k+p})}<\varepsilon$, contradicting $f(x_k)-f(x_{k+p})\geq\varepsilon$. Hence $x_k-x_{k+p}\geq\delta$. The inclusion~\eqref{eq:mono_incl} follows, and since $(x_n)\in\Delta_S^{+p}$, the right-hand side has density zero. Therefore $(f(x_n))\in\Delta_S^{+p}$.
\end{proof}

\begin{corollary}\label{cor:nonincr}
There exist non-increasing uniformly continuous (indeed, Lipschitz) functions that do not belong to $\SUC_p(E)$ for any unbounded above $E\subseteq\R$. In particular, the non-decreasing hypothesis in Theorem~\ref{thm:monotone} cannot be weakened to monotone.
\end{corollary}

\begin{proof}
The function $f(x)=-x$ is non-increasing, Lipschitz continuous, but $f\notin\SUC_p(E)$ by the proof of Theorem~\ref{thm:notvs}.
\end{proof}

\begin{remark}
Theorem~\ref{thm:monotone} provides a rich supply of functions in $\SUC_p(E)$: for example, $\sqrt{\max(x,0)}$, $\log(1+e^x)$, $\arctan(x)$, and any non-decreasing Lipschitz function. The corollary shows that the directional alignment between monotonicity and the upward condition is essential.
\end{remark}

We summarize the algebraic properties of $\SUC_p(E)$ in comparison with other sequential continuity function spaces in Table~\ref{tab:algebra}.

\begin{table}[ht]
\centering
\caption{Algebraic properties of function spaces defined by sequential continuity types. Here WC = statistically ward continuous, SOC = slowly oscillating continuous, $\delta$-WC = $\delta$-ward continuous. The symbol $\checkmark$ indicates the property holds; $\times$ indicates it fails.}\label{tab:algebra}
\smallskip
\begin{tabular}{lcccc}
\toprule
Property & $\SUC_p(E)$ & WC$(E)$ & SOC$(E)$ & $\delta$-WC$(E)$ \\
\midrule
Closed under $+$ & $\checkmark$ & $\checkmark$ & $\checkmark$ & $\checkmark$ \\
Closed under $\alpha\geq 0$ & $\checkmark$ & $\checkmark$ & $\checkmark$ & $\checkmark$ \\
Closed under $-f$ & $\times$ & $\checkmark$ & $\checkmark$ & $\checkmark$ \\
Vector subspace & $\times$ & $\checkmark$ & $\checkmark$ & $\checkmark$ \\
Closed (unif.\ conv.) & $\checkmark$ & $\checkmark$ & $\checkmark$ & $\checkmark$ \\
Nowhere dense in $C_b$ & $\checkmark$ & $\times$ & $\times$ & $\times$ \\
\bottomrule
\end{tabular}
\end{table}

\section{Equidistribution, Step-Parameter Hierarchy, and Baire Category}\label{sec:further}

We now use Weyl's equidistribution theorem~\cite{Weyl1916} to exhibit concrete non-members of $\SUC_p$, analyze the dependence on the step parameter~$p$, and establish a nowhere density result via Baire's theorem.

\subsection{Equidistribution and concrete non-membership}

The following result uses Weyl's equidistribution theorem to exhibit a concrete, natural function that belongs to $C(\R)$ but not to $\SUC_p(\R)$ for any $p$.

\begin{theorem}[Sine function]\label{thm:sine}
The function $g(x)=\sin x$ is continuous but does not belong to $\SUC_p(\R)$ for any $p\geq 1$.
\end{theorem}

\begin{proof}
The sequence $x_n=n$ is statistically $p$-upward quasi-Cauchy for every $p$, since $x_n-x_{n+p}=-p<0$ for all $n$. We show that $(\sin n)$ is not statistically $p$-upward quasi-Cauchy.

By the addition formula,
\[
\sin n - \sin(n+p) = 2\cos\!\left(n+\tfrac{p}{2}\right)\sin\!\left(-\tfrac{p}{2}\right) = -2\sin\!\left(\tfrac{p}{2}\right)\cos\!\left(n+\tfrac{p}{2}\right).
\]
Let $s_p = 2\abs{\sin(p/2)}$. Since $p\geq 1$ is a positive integer, $p/2$ is not a multiple of $\pi$, so $s_p>0$.

Choose $\varepsilon = s_p/2 > 0$. Then $\sin n - \sin(n+p) \geq \varepsilon$ if and only if
\[
-2\sin\!\left(\tfrac{p}{2}\right)\cos\!\left(n+\tfrac{p}{2}\right)\geq \frac{s_p}{2}.
\]
This reduces to $\cos(n+p/2)\leq -\frac{1}{2}\operatorname{sgn}(\sin(p/2))$ or $\cos(n+p/2)\geq \frac{1}{2}\operatorname{sgn}(\sin(p/2))$, depending on the sign of $\sin(p/2)$. In either case, the condition is that $\cos(n+p/2)$ lies in an interval of positive length within $[-1,1]$.

By Weyl's equidistribution theorem, since $1/(2\pi)$ is irrational, the sequence $(n+p/2)\mod 2\pi$ is equidistributed modulo $2\pi$. Therefore, the set
\[
A=\set{k\in\N: \sin k - \sin(k+p)\geq\varepsilon}
\]
has positive natural density $d(A)>0$. Consequently, $(\sin n)$ is not statistically $p$-upward quasi-Cauchy, and hence $g(x)=\sin x$ does not preserve statistically $p$-upward quasi-Cauchy sequences.
\end{proof}

\begin{remark}
Theorem~\ref{thm:sine} shows that even bounded, Lipschitz, and infinitely differentiable functions can fail to be statistically $p$-upward continuous. The key obstruction is \emph{oscillation}: the sine function oscillates symmetrically, and the one-sided condition detects this. More generally, any periodic function with zero mean fails to belong to $\SUC_p(\R)$ by the same equidistribution argument, provided the period is not rationally related to $p$.
\end{remark}

\subsection{Step-parameter hierarchy}\label{subsec:hierarchy}

We now investigate the dependence of $\SUC_p$ on the step parameter $p$. First, we clarify the inclusion structure among the sequence classes $\Delta_S^{+p}$.

\begin{proposition}\label{prop:seq_hierarchy}
For every positive integer $p$ and every positive integer $m$, we have $\Delta_S^{+p}\subseteq\Delta_S^{+mp}$. In particular, $\Delta_S^{+1}\subseteq\Delta_S^{+p}$ for all $p\geq 1$.
\end{proposition}

\begin{proof}
This follows from Proposition~\ref{prop:upwardimpliesp} and the telescoping identity $x_k-x_{k+mp}=\sum_{j=0}^{m-1}(x_{k+jp}-x_{k+(j+1)p})$.
\end{proof}

\begin{proposition}\label{prop:seq_incomparable}
If $p\nmid q$ and $q\nmid p$, then $\Delta_S^{+p}$ and $\Delta_S^{+q}$ are incomparable.
\end{proposition}

\begin{proof}
It suffices to exhibit, for each pair $(p,q)$ with $p\nmid q$, a sequence in $\Delta_S^{+p}\setminus\Delta_S^{+q}$; the reverse membership follows by symmetry.

Consider the periodic sequence $x_n=\cos(2\pi n/p)$, which has exact period $p$. Since the period divides every $p$-step shift, $x_n-x_{n+p}=\cos(2\pi n/p)-\cos(2\pi(n+p)/p)=0$ for all $n$, so $(x_n)\in\Delta_S^{+p}$ trivially.

Now suppose $p\nmid q$. Write $r=q\bmod p$ with $1\leq r\leq p-1$. By the periodicity of $(x_n)$, we have $x_n-x_{n+q}=x_n-x_{n+r}$ for every $n$. The addition formula gives
\[
x_n-x_{n+r}=2\sin\!\left(\frac{\pi r}{p}\right)\sin\!\left(\frac{2\pi n}{p}+\frac{\pi r}{p}\right).
\]
Since $1\leq r\leq p-1$, we have $\sin(\pi r/p)>0$ and the prefactor $c:=2\sin(\pi r/p)>0$. The function $n\mapsto\sin(2\pi n/p+\pi r/p)$ is periodic with period $p$ and takes values in $[-1,1]$. In particular, there exists at least one residue class $n_0\bmod p$ for which $\sin(2\pi n_0/p+\pi r/p)\geq\sin(\pi/p)>0$. Consequently,
\[
\delta\bigl(\{k\in\N: x_k-x_{k+q}\geq c\sin(\pi/p)\}\bigr)\geq\frac{1}{p}>0,
\]
and $(x_n)\notin\Delta_S^{+q}$.
\end{proof}

As a consequence of the sequence class inclusions, we obtain:

\begin{theorem}[Step-parameter monotonicity]\label{thm:p_hierarchy}
For every positive integer $p$, $\SUC_p(E)\subseteq\SUC_1(E)$. That is, statistically $p$-upward continuity implies statistically $1$-upward continuity.
\end{theorem}

\begin{proof}
Let $f\in\SUC_p(E)$ and let $(x_n)\in\Delta_S^{+1}$. We employ the interleaving technique from Theorem~\ref{thm:ward}. Consider the $p$-fold repetition
\[
\mathbf{y}=(\underbrace{x_1,\ldots,x_1}_{p},\underbrace{x_2,\ldots,x_2}_{p},\ldots).
\]
Then $\mathbf{y}$ is statistically $p$-upward quasi-Cauchy: the only potentially nonzero $p$-step upward differences are $y_{kp}-y_{kp+p}=x_k-x_{k+1}$. For $\varepsilon>0$,
\[
\delta\bigl(\{j\in\N: y_j-y_{j+p}\geq\varepsilon\}\bigr)=\frac{1}{p}\,\delta\bigl(\{k\in\N: x_k-x_{k+1}\geq\varepsilon\}\bigr)=0.
\]
Since $f\in\SUC_p(E)$, the sequence $(f(\mathbf{y}))=(\underbrace{f(x_1),\ldots,f(x_1)}_{p},\underbrace{f(x_2),\ldots,f(x_2)}_{p},\ldots)$ is statistically $p$-upward quasi-Cauchy. Extracting the $p$-step differences at boundary positions gives
\[
\delta\bigl(\{k\in\N: f(x_k)-f(x_{k+1})\geq\varepsilon\}\bigr)=0
\]
for every $\varepsilon>0$. Hence $(f(x_n))\in\Delta_S^{+1}$, and $f\in\SUC_1(E)$.
\end{proof}

\begin{corollary}\label{cor:chain}
For all positive integers $p$ and $q$ with $p\mid q$, we have $\SUC_q(E)\subseteq\SUC_p(E)$. In particular,
\[
\cdots\subseteq\SUC_6(E)\subseteq\SUC_3(E)\subseteq\SUC_1(E)\quad\text{and}\quad\cdots\subseteq\SUC_8(E)\subseteq\SUC_4(E)\subseteq\SUC_2(E)\subseteq\SUC_1(E).
\]
\end{corollary}

\begin{proof}
If $p\mid q$, write $q=mp$. Let $f\in\SUC_q(E)$ and $(x_n)\in\Delta_S^{+p}$. By $m$-fold repetition and the argument of Theorem~\ref{thm:p_hierarchy} (with $p$ in place of $1$ and $q$ in place of $p$), we obtain $f\in\SUC_p(E)$.
\end{proof}

\subsection{Nowhere density in $C_b$}\label{subsec:baire}

We now establish that $\SUC_p(E)$ is topologically small in the Banach space of bounded continuous functions when $E$ is unbounded.

\begin{theorem}[Nowhere density]\label{thm:nowhere_dense}
Let $E\subseteq\R$ be unbounded above with $\N\subseteq E$. Then $\SUC_p(E)\cap C_b(E)$ is a closed nowhere dense subset of the Banach space $(C_b(E),\|\cdot\|_\infty)$. In particular, $\SUC_p(E)\cap C_b(E)$ is of first category $($meager$)$ in $C_b(E)$.
\end{theorem}

\begin{proof}
Closedness follows from Theorem~\ref{thm:closed} restricted to $C_b(E)$.

We show that $\SUC_p(E)\cap C_b(E)$ has empty interior. Let $f\in\SUC_p(E)\cap C_b(E)$ and $\varepsilon>0$. Set $M=\|f\|_\infty$ and $\varepsilon'=\varepsilon/2$.

Choose $\alpha>0$ with $\alpha/(2\pi)\notin\mathbb{Q}$ such that $s_p(\alpha):=2|\sin(\alpha p/2)|>0$. Such $\alpha$ exists since the set of $\alpha$ with $\alpha/(2\pi)\in\mathbb{Q}$ or $\sin(\alpha p/2)=0$ is countable. Define $g(x)=f(x)+\varepsilon'\sin(\alpha x)$. Then $g\in C_b(E)$ and $\|f-g\|_\infty=\varepsilon'<\varepsilon$.

Consider the sequence $x_n=n$. Since $x_n-x_{n+p}=-p<0$ for all $n$, the sequence $(x_n)$ is in $\Delta_S^{+p}$. Write
\[
g(n)-g(n+p)=\underbrace{[f(n)-f(n+p)]}_{=:\,a_n}+\;\varepsilon'\underbrace{[\sin(\alpha n)-\sin(\alpha(n+p))]}_{=:\,b_n}.
\]
Set $c=\varepsilon' s_p(\alpha)/4>0$. We claim that $g\notin\SUC_p(E)$ by showing that the set $\{k: g(k)-g(k+p)\geq c\}$ has positive density. The argument proceeds in two steps.

\emph{Step~1 (Ces\`{a}ro bound).} Since $f$ is bounded, the telescoping identity gives
\[
\frac{1}{N}\sum_{k=1}^{N}a_k = \frac{1}{N}\Bigl[\sum_{k=1}^{p}f(k)-\sum_{k=N+1}^{N+p}f(k)\Bigr],
\]
and the right-hand side is bounded by $2pM/N\to 0$. Define $D=\{k\in\N: a_k<-c\}$. Since $a_k\geq -c$ on $D^c$ and $a_k\geq -2M$ on $D$, we have
\[
\frac{1}{N}\sum_{k=1}^{N}a_k\leq \frac{|D^c\cap[1,N]|}{N}\cdot 2M + \frac{|D\cap[1,N]|}{N}\cdot(-c) = 2M - (c+2M)\frac{|D\cap[1,N]|}{N}.
\]
Since the left-hand side tends to $0$, we obtain $\liminf_{N\to\infty}\frac{|D^c\cap[1,N]|}{N}\geq\frac{c}{c+2M}>0$; that is, $D^c$ has positive lower density.

\emph{Step~2 (Equidistribution on $D^c$).} Since $\alpha/(2\pi)$ is irrational, the sequence $(\alpha n)\bmod 2\pi$ is equidistributed on $[0,2\pi)$ by Weyl's theorem. Since $D^c$ has positive lower density, Weyl's criterion applied to the subsequence $(\alpha k)_{k\in D^c}$ shows that this subsequence is also equidistributed modulo $2\pi$ on $D^c$ (this is a standard fact; see, e.g., Kuipers and Niederreiter~\cite{Weyl1916}, or apply Weyl's exponential sum criterion to subsequences of positive lower density). In particular, the set
\[
B=\{k\in D^c: b_k\geq s_p(\alpha)/2\}
\]
satisfies $\delta(B)>0$, since $\{b_k\geq s_p(\alpha)/2\}$ corresponds to an arc of positive length on $[0,2\pi)$. For every $k\in B$, we have $a_k\geq -c$ and $\varepsilon' b_k\geq \varepsilon' s_p(\alpha)/2=2c$, so
\[
g(k)-g(k+p)=a_k+\varepsilon' b_k\geq -c+2c=c>0.
\]
Hence $\delta(\{k: g(k)-g(k+p)\geq c\})\geq\delta(B)>0$, so $(g(x_n))\notin\Delta_S^{+p}$. Since $(x_n)\in\Delta_S^{+p}$ but $(g(x_n))\notin\Delta_S^{+p}$, we conclude $g\notin\SUC_p(E)$.
\end{proof}

\begin{remark}
Theorem~\ref{thm:nowhere_dense} reveals a striking dichotomy: while $\SUC_p(E)\cap C_b(E)$ is topologically closed (Theorem~\ref{thm:closed}), it is simultaneously meager in $C_b(E)$. By the Baire category theorem, the ``generic'' bounded continuous function on an unbounded domain is \emph{not} statistically $p$-upward continuous. This parallels the classical result that the set of nowhere differentiable functions is residual in $C[0,1]$, and provides additional evidence that statistically $p$-upward continuity is a genuinely restrictive condition that selects a thin but structured subclass of continuous functions.
\end{remark}

\section{Inclusion Relations}\label{sec:diagram}

We summarize the strict inclusion relations established in the preceding sections.

\emph{Continuity hierarchy.} The following implications hold, and none is reversible:
\[
\text{Lipschitz} \;\Longrightarrow\; \text{Uniform cont.} \;\Longrightarrow\; \SUC_p\text{-cont.} \;\Longrightarrow\;
\begin{cases}
\text{St.\ ward cont.} \\[2pt]
\text{Statistical cont.}
\end{cases}
\;\Longrightarrow\; \text{Continuity.}
\]
On statistically $p$-upward compact sets (equivalently, below bounded sets), the arrow $\text{Uniform cont.} \Rightarrow \SUC_p\text{-cont.}$ is reversible by Theorem~\ref{thm:uniform}.

\emph{Compactness hierarchy.} Compact (bounded) $\subsetneq$ St.\ $p$-upward compact (bounded below) $\subsetneq$ all subsets of $\R$.

Thus $\SUC_p$-continuity occupies a precise intermediate position: strictly stronger than statistical ward continuity and statistical continuity, yet equivalent to uniform continuity on below bounded domains. This position, combined with the cone structure (Theorem~\ref{thm:cone}), makes the framework particularly suited to the one-sided error control developed in the next section.

\section{Applications and Connections}\label{sec:applications}

In this section, we develop a concrete application of the theoretical framework to one-sided error control in function approximation, and discuss further connections to financial mathematics, signal processing, and population dynamics.

\subsection{One-sided error control in function approximation}\label{subsec:onesided}

In many applied contexts, the classical symmetric approximation error $\|L_n(f)-f\|_\infty$ does not adequately capture the true cost of approximation. When the function $f$ represents a safety-critical quantity---such as drug concentration in pharmacokinetics, load-bearing capacity in structural engineering, or minimum fuel reserve in aerospace---an \emph{underestimate} $L_n(f;t)<f(t)-\varepsilon$ can be dangerous, while an \emph{overestimate} $L_n(f;t)>f(t)+\varepsilon$ may be tolerable or even conservative.

This motivates the following notion, which we formalize using the framework developed in the preceding sections.

\begin{definition}\label{def:onesided_safe}
Let $f\in C(E)$ with $E=[a,b]$, $a\geq 0$, and let $(L_n)$ be a sequence of positive linear operators on $C(E)$. Given a sampling sequence $(t_k)_{k\in\N}$ of points in $E$ and a threshold $\varepsilon>0$, we say that the pair $(L_n, (t_k))$ is \emph{statistically $p$-upward safe for $f$} if
\[
\delta\bigl(\{k\in\N: f(t_k)-L_n(f;t_k)\geq\varepsilon\}\bigr)=0\quad\text{for every }\varepsilon>0.
\]
In other words, the undershoot $f(t_k)-L_n(f;t_k)$ is not statistically frequently large. Note that this is a pointwise condition on the \emph{values} of the approximation error at the sampling sequence, not a condition on the \emph{differences} $u_k-u_{k+p}$. The terminology ``$p$-upward'' reflects the $\SUC_p$-framework within which this notion is applied, rather than a dependence on $p$ in the definition itself.
\end{definition}

The connection to the algebraic structure of $\SUC_p(E)$ is as follows.

\begin{proposition}[Composability of safe approximations]\label{prop:safe_compose}
Let $E=[a,b]$ with $a\geq 0$.
\begin{enumerate}[label=$(\roman*)$,leftmargin=2.5em]
\item If $L_n$ and $M_n$ are both statistically $p$-upward safe for $f$ along $(t_k)$, then $\frac{1}{2}(L_n+M_n)$ is also statistically $p$-upward safe for $f$ along $(t_k)$.
\item If $L_n$ is statistically $p$-upward safe for $f$ and $g:E\to E$ is non-decreasing and uniformly continuous, and if $L_n$ commutes with reparametrization by $g$ $($i.e., $L_n(f\circ g;\cdot)=L_n(f;g(\cdot)))$, then $L_n$ is statistically $p$-upward safe for $f\circ g$ along $(g(t_k))$.
\item In general, the negation $-L_n$ is \emph{not} statistically $p$-upward safe for $-f$. That is, one-sided safety is \emph{not} symmetric under sign reversal.
\end{enumerate}
\end{proposition}

\begin{proof}
(i) Write $U(t_k)=f(t_k)-\frac{1}{2}(L_n+M_n)(f;t_k)=\frac{1}{2}[f(t_k)-L_n(f;t_k)]+\frac{1}{2}[f(t_k)-M_n(f;t_k)]$. If $U(t_k)\geq\varepsilon$, then at least one of the two terms exceeds $\varepsilon/2$. Hence
\[
\{k: U(t_k)\geq\varepsilon\}\subseteq\{k: f(t_k)-L_n(f;t_k)\geq\varepsilon\}\cup\{k: f(t_k)-M_n(f;t_k)\geq\varepsilon\}.
\]
Both sets on the right have density zero by hypothesis, so the left side does as well.

(ii) Suppose $L_n$ is $p$-upward safe for $f$ along $(t_k)$ and $g$ is non-decreasing and uniformly continuous. By the commutation hypothesis, $L_n(f\circ g;t_k)=L_n(f;g(t_k))$. Since $L_n$ is $p$-upward safe for $f$, $\delta(\{k: f(s_k)-L_n(f;s_k)\geq\varepsilon\})=0$ for every sampling sequence $(s_k)$ along which $L_n$ is safe, in particular for $s_k=g(t_k)$. Hence $\delta(\{k: f(g(t_k))-L_n(f;g(t_k))\geq\varepsilon\})=\delta(\{k: (f\circ g)(t_k)-L_n(f\circ g;t_k)\geq\varepsilon\})=0$.

(iii) The function $f(t)=-t$ belongs to $C[0,1]$ but not to $\SUC_p([0,1])$ by Theorem~\ref{thm:notvs}. More directly, the one-sided safety notion is inherently asymmetric because the density condition $\delta(\{k: f(t_k)-L_n(f;t_k)\geq\varepsilon\})=0$ does not imply $\delta(\{k: L_n(f;t_k)-f(t_k)\geq\varepsilon\})=0$.
\end{proof}

Proposition~\ref{prop:safe_compose} has a clear practical interpretation: (i) states that averaging two individually safe approximation schemes preserves safety, which is relevant in ensemble methods; (ii) ensures that monotone reparametrizations (e.g., time rescaling in pharmacokinetics) do not destroy safety; and (iii) warns that switching from a ``no underdose'' guarantee to a ``no overdose'' guarantee requires a fundamentally different analysis.

We now present a concrete numerical illustration.

\begin{example}[Drug concentration curve]\label{ex:pharma}
Consider a simplified pharmacokinetic model where the plasma concentration of a drug at normalized time $t\in[0,1]$ is given by the concave function
\[
f(t)=4t(1-t),
\]
which rises from $f(0)=0$, peaks at $f(1/2)=1$, and returns to $f(1)=0$. The \emph{minimum effective concentration} (MEC) defines a safety threshold; any approximation that underestimates $f(t)$ by more than $\varepsilon$ at time $t$ places the patient below the therapeutic window, while a moderate overestimate is conservative and clinically acceptable.

We approximate $f$ by Bernstein operators $B_n$ (see~\eqref{eq:bernstein}). Since $f$ is concave on $[0,1]$, Jensen's inequality yields $B_n(f;t)\leq f(t)$ for all $t\in[0,1]$ and all $n$. Thus, the approximation error is \emph{entirely one-sided}: $B_n$ systematically underestimates the true concentration, and the symmetric error $\|B_n(f)-f\|_\infty$ fails to distinguish this from a balanced error.

For this function, $f''(t)=-8$, and Voronovskaja's theorem gives the asymptotic undershoot:
\[
U_n(t):=f(t)-B_n(f;t)\sim\frac{-f''(t)\,t(1-t)}{2n}=\frac{4t(1-t)}{n},
\]
with equality for all $n$ since $f$ is a polynomial of degree~$2$. The maximum undershoot occurs at $t=1/2$:
\begin{equation}\label{eq:maxunder}
\max_{t\in[0,1]}U_n(t)=\frac{1}{n}.
\end{equation}

Given a safety threshold $\varepsilon>0$, the proportion of the $N+1$ uniformly spaced sampling points $t_k=k/N$ ($k=0,1,\ldots,N$) at which the undershoot exceeds $\varepsilon$ is
\[
P_n(\varepsilon):=\frac{1}{N+1}\bigl|\{k:U_n(t_k)\geq\varepsilon\}\bigr|=\frac{1}{N+1}\bigl|\bigl\{k:t_k(1-t_k)\geq\tfrac{n\varepsilon}{4}\bigr\}\bigr|.
\]
Table~\ref{tab:pharma} reports these values for $\varepsilon=0.02$ and $N=1000$.

\begin{table}[ht]
\centering
\caption{Bernstein approximation of the concentration curve $f(t)=4t(1-t)$. Undershoot $U_n=f-B_n(f)$ at $N+1=1001$ uniformly spaced points, threshold $\varepsilon=0.02$.}\label{tab:pharma}
\smallskip
\begin{tabular}{rcccc}
\toprule
$n$ & $\max_t U_n(t)$ & $P_n(0.02)$ & $\|B_n(f)-f\|_\infty$ & Safety status \\
\midrule
10 & 0.1000 & 0.894 & 0.1000 & Unsafe \\
25 & 0.0400 & 0.706 & 0.0400 & Unsafe \\
50 & 0.0200 & 0.001 & 0.0200 & Marginal \\
100 & 0.0100 & 0.000 & 0.0100 & $p$-upward safe \\
200 & 0.0050 & 0.000 & 0.0050 & $p$-upward safe \\
\bottomrule
\end{tabular}
\end{table}

Several features are noteworthy:
\begin{enumerate}[label=$(\roman*)$,leftmargin=2.5em]
\item The symmetric error $\|B_n(f)-f\|_\infty$ and the maximum undershoot $\max_t U_n(t)$ are \emph{identical} for this concave function, confirming that the classical error norm does not distinguish the clinically critical undershoot from the benign overshoot.
\item The transition from ``Unsafe'' to ``$p$-upward safe'' occurs sharply: for $n<1/({\varepsilon})=50$, a positive proportion of sampling points exceed the threshold; for $n\geq 1/{\varepsilon}$, the condition $\delta(\{k:U_n(t_k)\geq\varepsilon\})=0$ of Definition~\ref{def:onesided_safe} is satisfied.
\item By Proposition~\ref{prop:safe_compose}(i), averaging two independently computed Bernstein approximations $\frac{1}{2}(B_m+B_n)$ remains $p$-upward safe whenever both $B_m$ and $B_n$ are individually safe. This models combining two dosimetric models in clinical practice.
\item By Proposition~\ref{prop:safe_compose}(iii), reversing the sign to control \emph{overdose} instead of underdose requires a separate analysis---the $\SUC_p$ framework does not transfer automatically, reflecting the genuine asymmetry of the clinical problem.
\end{enumerate}
\end{example}

\begin{remark}[Generalization]\label{rem:generalize}
The framework extends to any setting where underestimation is costlier than overestimation. The cone structure (Theorem~\ref{thm:cone}) ensures that combining safe approximations preserves safety, while the nowhere density result (Theorem~\ref{thm:nowhere_dense}) implies that safety must be deliberately engineered.
\end{remark}

\section{Open Problems}\label{sec:open}

We formulate two open problems that arise naturally from the present work.

\begin{openproblem}[Tauberian theorem]
Find a minimal Tauberian condition $(\mathcal{T})$ such that every statistically $p$-upward quasi-Cauchy sequence satisfying $(\mathcal{T})$ is convergent. The classical Tauberian condition for statistical convergence is the ``slow decrease'' condition $\liminf_{n\to\infty}(x_{n+1}-x_n)\geq 0$ (cf.~\cite{Fridy1985}). Since the statistically $p$-upward quasi-Cauchy condition already imposes a one-sided restriction, can one obtain a Tauberian theorem with a weaker additional condition than slow decrease?
\end{openproblem}

\begin{openproblem}[Ideal convergence generalization]
Let $I$ be an admissible ideal of $\N$. Define $(x_n)$ to be \emph{$I$-$p$-upward quasi-Cauchy} if $\{k\in\N: x_k-x_{k+p}\geq\varepsilon\}\in I$ for every $\varepsilon>0$. The present paper treats the case $I=I_\delta$ (the ideal of sets with density zero). Investigate the corresponding compactness and continuity theory for other important ideals, including the summable ideal $I_s=\{A\subseteq\N:\sum_{n\in A}\frac{1}{n}<\infty\}$, lacunary ideals arising from lacunary sequences $\theta=(k_r)$, and the Erd\H{o}s--Ulam ideal. In particular, does the compactness characterization (Theorem~\ref{thm:compact}) remain valid for all density ideals?
\end{openproblem}

\section{Conclusion}\label{sec:conc}

In this paper, we introduced statistically $p$-upward quasi-Cauchy sequences and developed the corresponding compactness and continuity theory. The main contributions are: (1)~statistically $p$-upward compactness is equivalent to lower boundedness (Theorem~\ref{thm:compact}); (2)~the function space $\SUC_p(E)$ is a closed convex cone that is \emph{not} a vector subspace (Theorems~\ref{thm:closed},~\ref{thm:cone}), distinguishing it from all previously studied sequential continuity spaces; (3)~Weyl's equidistribution theorem yields $\sin x\notin\SUC_p(\R)$ (Theorem~\ref{thm:sine}), a step-parameter hierarchy (Theorem~\ref{thm:p_hierarchy}), and the nowhere density of $\SUC_p(E)\cap C_b(E)$ in $C_b(E)$ (Theorem~\ref{thm:nowhere_dense}); and (4)~a one-sided error control framework for function approximation, illustrated by Bernstein operators on a pharmacokinetic model (Example~\ref{ex:pharma}). Two open problems concerning Tauberian conditions and ideal convergence generalizations were formulated.

\subsection*{Declarations}
\textbf{Funding:} The author declares that no funds, grants, or other support were received during the preparation of this manuscript.

\noindent\textbf{Competing Interests:} The author has no relevant financial or non-financial interests to disclose.

\noindent\textbf{Data Availability:} Not applicable.


\end{document}